\documentclass[preprint]{elsarticle}

\usepackage{graphicx}
\usepackage{epstopdf}

\usepackage{amssymb}
\usepackage{amsthm}
\usepackage{amsmath}
\usepackage{url}
\usepackage{hyperref}
\usepackage{multirow}
\usepackage{epstopdf}
\usepackage{booktabs}
\usepackage{mathrsfs}

\journal{Elsevier}

\begin{document}

\begin{frontmatter}



%\title{Numerical solution of fractional sub-diffusion equation by an exponential B-spline collocation method}
\title{An exponential B-spline collocation method for fractional sub-diffusion equation}

\cortext[cor1]{Corresponding author}
\author{X. G. Zhu}
\author{Y. F. Nie\corref{cor1}}
\ead{yfnie@nwpu.edu.cn}
\author{Z. B. Yuan}
\author{J. G. Wang}
\author{Z. Z. Yang}
%\author[focal]{D. L. Wang}

\address{Department of Applied Mathematics, Northwestern Polytechnical University, Xi' an 710129, P.R. China}
%\address[focal]{Department of Mathematics, Northwest University, Xi' an 710069, P.R. China}

\begin{abstract}
In this article, we propose an exponential B-spline collocation method to approximate the solution of
the fractional sub-diffusion equation of Caputo type. The present method is generated by use of
the Gorenflo-Mainardi-Moretti-Paradisi (GMMP) scheme in time and an efficient exponential B-spline based method in space.
The unique solvability is rigorously discussed. Its stability is well illustrated via a procedure closely resembling the classic von Neumann approach.
The resulting algebraic system is tri-diagonal that can rapidly be solved by the known algebraic solver with low cost and storage.\
A series of numerical examples are finally carried out and by contrast to the other algorithms available in the literature,
numerical results confirm the validity and superiority of our method.
\end{abstract}

\begin{keyword}
Fractional sub-diffusion equation, GMMP scheme, Exponential B-spline collocation method, Solvability and stable analysis.


\end{keyword}

\end{frontmatter}


\newtheorem{theorem}{Theorem}[section]
\newtheorem{lemma}{Lemma}[section]
\newtheorem{definition}{Defition}[section]
\newtheorem{remark}{Remarks}[section]
\newtheorem{assu}{Assumption}[section]
\renewcommand{\theequation}{\arabic{section}.\arabic{equation}}

\section{Introduction}\label{s1}
The basic concept of anomalous diffusion  dates back to Richardson's treatise on atmospheric diffusion in 1926 \cite{Ref001}.
It has increasingly got recognition since the late 1960s within transport theory. In contrast to a typical diffusion, such process no longer follows Gaussian statistics,
then the classic Fick's law fails to apply. Its most striking character is the temporal
power-law pattern dependence of the mean squared displacement \cite{Ref002}, i.e., $\chi^2(t)\sim \kappa t^\alpha$,
for sub-diffusion, $\alpha<1$, while $\alpha>1$ for super-diffusion. Anomalous transport behavior is ubiquitous in physical
scenarios and due to its universal mutuality,  formidable challenges are introduced. % in physics and mathematics.
In recent decades, fractional partial differential equations (PDEs) enter public vision that compare favorably with the usual models to characterize
such transport motions in heterogeneous aquifer and the medium with fractal geometry \cite{Ref003,Ref004}.
An explosive interest has been gained among academic circles to scramble to investigate the theoretical properties, analytic techniques,
and numerical algorithms for fractional PDEs \cite{Ref005,Ref038,Ref006,Ref040,Ref007,Ref068,Ref039,Ref010}.

As a model problem of the class of fractional PDEs described above,
the fractional sub-diffusion equation is  considered here
\begin{align}
\frac{\partial^\alpha u(x,t)}{\partial t^\alpha}-\kappa\frac{\partial^2 u(x,t)}{\partial x^2} =f(x,t),
    \quad a< x< b, \ 0<t\leq T,\label{eq01}
\end{align}
subjected to the initial and boundary conditions as
\begin{align}
  & u(x,0)=\varphi(x),\quad a\leq x\leq b, \label{eq02}\\
  & u(a,t)=g_1(t), \quad u(b,t)=g_2(t), \quad 0<t\leq T, \label{eq03}
\end{align}
where $0<\alpha<1$, $\kappa$ is the positive viscosity constant, and $\varphi(x)$, $g_1(t)$, $g_2(t)$ are the prescribed functions
with sufficient smoothness. In Eq. (\ref{eq01}), the time-fractional derivative is defined in Caputo sense, i.e.,
\begin{align*}%{^C_0}D^\alpha_tu(x,t)=
 \frac{\partial^\alpha u(x,t)}{\partial t^\alpha}=\frac{1}{\Gamma(1-\alpha)}
  \int^t_0\frac{\partial u(x,\xi)}{\partial \xi}\frac{d\xi}{(t-\xi)^\alpha},
\end{align*}
with the Gamma function $\Gamma(\cdot)$.  %a large number of Euler's
There have already been some works dedicated to develop numerical algorithms to solve Eqs. (\ref{eq01})-(\ref{eq03}) apart from
a few analytic methods that are not always available for general situations.
Zhang and Liu derived an implicit difference scheme and proved that it is unconditional stable \cite{Ref015}.
Yuste and Acedo studied an explicit difference scheme based on Gr\"{u}nwald-Letnikov formula \cite{Ref014}.
Along the same line, a group of weighted average difference schemes was then obtained \cite{Ref019}.
In \cite{Ref023}, Cui raised a high-order compact difference scheme and
%its stability and convergence were detailedly discussed; another similar approach was the compact scheme stated in \cite{Ref020},
its convergence was detailedly discussed; another similar approach was the compact scheme stated in \cite{Ref020},
for the fractional sub-diffusion equation with Neumann boundary condition.
In \cite{Ref012}, an effective spectral method was constructed by using the common
$L^1$-formula in time and a Legendre spectral approximation in space. Later, this method was extended
to the time-space case \cite{Ref021}. The finite element method was considered by Jiang and Ma \cite{Ref011}.
%also with $L^1$ approximation in time.
The semi-discrete lump finite element method was studied by Jin et al.\ for a time-fractional model with a nonsmooth right-hand side \cite{Ref022}.
%where an optimal error estimate requires symmetric meshes was obtained.  described derived  approximation raised  stated
Liu et al.\ described an implicit RBF meshless approach for the time-fractional diffusion equation \cite{Ref016}. Li et al.\ suggested an adomian decomposition algorithm
for the equations of the same type \cite{Ref009}. In \cite{Ref017}, the authors solved such a model by
the direct discontinuous Galerkin method with the Caputo derivative discretized by a GMMP scheme.
Recently, Luo et al.\ established a quadratic spline collocation method for the fractional sub-diffusion equation \cite{Ref024}, where the convergence
under $L^\infty$-norm was analyzed. Sayevand et al.\ conducted a cubic B-spline collocation method \cite{Ref025}, whose
stability was provided as well. In \cite{Ref026}, a Sinc-Haar collocation method was proposed, which used the Haar operational matrix to convert
the original problem into linear algebraic equations. %via expanding the approximation based on Sinc and Haar functions.

In the present work, regarding the current interest in efficient numerical algorithms for fractional PDEs,
we showcase a collocation method based on exponential B-spline trial function to solve Eqs. (\ref{eq01})-(\ref{eq03}).
The Caputo derivative is tackled by GMMP formula and the spatial derivative is approximated in an exponential spline space
via a uniform nodal collocation strategy. A von Neumann like procedure leads to its unconditional stability.
Its codes are tested on five numerical examples and studied in contrast with the other algorithms.
The obtained method is highly accurate and calls for a lower cost to implement. This may make sense to treat
the equations as the model we consider here with a long time range.
The outline is as follows. In Section \ref{s2}, we give a concise description of the exponential B-spline trial basis,
which will be useful hereinafter. In Section \ref{s3}, we construct a fully discrete exponential B-spline method on uniform meshes
to discretize the model and prove that it is stable. The initial vector is addressed in Section \ref{s4},
which we require to start our method. To evaluate its accuracy and advantages, numerical examples are covered in Section \ref{s5}.

\section{Description of exponential spline functions}\label{s2}
In the sequel, let $a=x_0<x_1<x_2<\cdots< x_{M-1}<x_M=b$ be an equidistant spatial mesh on
the interval $[a,b]$, and for  $M\in\mathbb{N}^+$, denote
\begin{align*}
 h=(b-a)/M, \quad  p=\max\limits_{1\leq j\leq M}p_j,
 \quad s=\sinh(ph), \quad c=\cosh(ph),
\end{align*}
where $p_j$ is the value of function $p(x)$ at mesh knot $x_j$.
The exponential splines are a kind of piecewise non-polynomial functions that are known as a generalization of the semi-classical cubic splines.
They are recognized as a continuum of interpolants ranging from the cubic splines to the linear cases \cite{Ref027}.
Also, like the polynomial splines, a basis of exponential B-splines is admitted and an advisable definition
is the one introduced by McCartin \cite{Ref028}, each of which is support on finite subsegments.
On the above mesh together with another six knots $x_j$, $j=-3,-2,-1,M+1,M+2,M+3$ beyond $[a,b]$, the mentioned exponential B-splines $B_j(x)$,
$j=-1,0,\ldots,M+1$, are given as follows
\begin{align*}
B_j(x)=\left\{
\begin{aligned}
&e(x_{j-2}-x)-\frac{e}{p}\sinh(p(x_{j-2}-x)), \, \qquad\quad\qquad\qquad\qquad\textrm{if} \ \ x\in [x_{j-2},x_{j-1}], \\
&a+b(x_j-x)+c\exp(p(x_j-x))+d\exp(-p(x_j-x)), \quad\textrm{if} \ \ x\in [x_{j-1},x_{j}], \\
&a+b(x-x_j)+c\exp(p(x-x_j))+d\exp(-p(x-x_j)), \quad\textrm{if} \ \ x\in [x_{j},x_{j+1}], \\
&e(x-x_{j+2})-\frac{e}{p}\sinh(p(x-x_{j+2})), \, \qquad\quad\qquad\qquad\qquad\textrm{if} \ \ x\in [x_{j+1},x_{j+2}], \\
&0, \quad \textrm{otherwise},
\end{aligned}
\right.
\end{align*}
where
\begin{gather*}
e=\frac{p}{2(phc-s)},\quad a=\frac{phc}{phc-s},\quad b=\frac{p}{2}\bigg[\frac{c(c-1)+s^2}{(phc-s)(1-c)}\bigg],\\
c=\frac{1}{4}\bigg[\frac{\exp(-ph)(1-c)+s(\exp(-ph)-1)}{(phc-s)(1-c)}\bigg], \quad
d=\frac{1}{4}\bigg[\frac{\exp(ph)(c-1)+s(\exp(ph)-1)}{(phc-s)(1-c)}\bigg].
\end{gather*}
The values of $B_j(x)$ at each knot are given as
\begin{align}\label{eq11}
B_j(x_k)=\left\{
\begin{aligned}
&1,  &\textrm{if}\ \ k=j, \\
&\frac{s-ph}{2(phc-s)},  &\textrm{if} \ \ k=j\pm1,\\
&0,  &\textrm{if} \ \ k=j\pm2.
\end{aligned}
\right.
\end{align}
The values of $B'_j(x)$ and $B''_j(x)$ at each knot are given as
\begin{align}\label{eq12}
B'_j(x_k)=\left\{
\begin{aligned}
&0  &\textrm{if}\ \ k=j, \\
&\frac{\mp p(1-c)}{2(phc-s)},  &\textrm{if} \ \ k=j\pm1,\\
&0,  &\textrm{if} \ \ k=j\pm2,
\end{aligned}
\right.
\end{align}
and
\begin{align}\label{eq13}
B''_j(x_k)=\left\{
\begin{aligned}
&\frac{-p^2s}{phc-s},  &\textrm{if}\ \ k=j, \\
&\frac{p^2s}{2(phc-s)},  &\textrm{if} \ \ k=j\pm1,\\
&0,  &\textrm{if} \ \ k=j\pm2.
\end{aligned}
\right.
\end{align}

The set of $B_j(x)\in C^2(\mathbb{R})$, $j=-1,0,\ldots,M+1$, are linearly independent and form an exponential
spline space on $[a,b]$. The non-negative free $p$ is termed ``tension" parameter and $p\rightarrow 0$ yields cubic spline whereas
$p\rightarrow \infty$ corresponds to the linear spline. The cubic spline interpolation causes extraneous
inflexion points while the exponential spline interpolation allows to remedy this issue.

%and contains cubic splines as special case.
%A particular member of this family is uniquely specified by

\section{An exponential B-spline collocation method}\label{s3}
Let $t_n=n\tau$, $n=0,1,\ldots,N$, $T=\tau N$, $N\in\mathbb{N}^+$, and $x_j=a+jh$, $j=-1,0,\ldots,M+1$, $h=(b-a)/M$, $M\in\mathbb{N}^+$.
On the time-space lattice, we set about deriving the exponential B-spline collocation method for  Eqs. (\ref{eq01})-(\ref{eq03}).

\subsection{GMMP scheme for Caputo derivative}
To start with, we recall the Caputo and Riemann-Liouville fractional derivatives.
Given a smooth enough $f(x,t)$, the $\alpha$-th Caputo derivative is defined by
\begin{align}\label{eq04}
 {^C_0}D^\alpha_tf(x,t)=\frac{1}{\Gamma(m-\alpha)}
  \int^t_0\frac{\partial^m f(x,\xi)}{\partial \xi^m}\frac{d\xi}{(t-\xi)^{1+\alpha-m}},
\end{align}
and the $\alpha$-th Riemann-Liouville type derivative is defined by
\begin{align}\label{eq05}
 {^{RL}_0}D^\alpha_tf(x,t)=\frac{1}{\Gamma(m-\alpha)}\frac{\partial^m}{\partial t^m}
  \int^t_0\frac{f(x,\xi)d\xi}{(t-\xi)^{1+\alpha-m}},
\end{align}
where, $m-1<\alpha<m$, $m\in\mathbb{N}$ is not less than $1$. In common sense, (\ref{eq04}) owns merits in handling the initial-valued problems,
and thereby is utilized in time in most instances. (\ref{eq04}), (\ref{eq05}) interconvert into each other through %a relation
\begin{equation}\label{eq06}
    {^C_0}D^\alpha_tf(x,t)={^{RL}_0}D^\alpha_tf(x,t)-\sum^{m-1}_{l=0}\frac{f^{(l)}(x,0)t^{l-\alpha}}{\Gamma(l+1-\alpha)}.
\end{equation}
They are equal when $f^{(k)}(x,0)=0$, $k=0,1,\ldots,m-1$ are fixed; we refer the readers to \cite{Ref069,Ref032} for deeper insight.
A GMMP scheme is derived by rewriting Eq. (\ref{eq06}) and using a proper scheme to discretize (\ref{eq05}), which reads \cite{Ref030}
\begin{equation}\label{eq07}
   {^C_0}D^\alpha_tf(x,t_n)\approx \frac{1}{\tau^\alpha}\sum_{k=0}^{n}\omega^\alpha_kf(x,t_{n-k})
   -\frac{1}{\tau^\alpha}\sum^{m-1}_{l=0}\sum^{n}_{k=0}\frac{\omega^\alpha_k f^{(l)}(x,0)t_{n-k}^{l}}{l!},
\end{equation}
with several valid sets of coefficients $\omega^\alpha_k$ \cite{Ref014}. In particular, when
\begin{equation}\label{eq08}
   \omega^\alpha_k=(-1)^k\binom\alpha k=\frac{\Gamma{(k-\alpha)}}{\Gamma{(-\alpha)}\Gamma{(k+1)}}, \quad k=0,1,2,\ldots
\end{equation}
it is the one given by Gorenflo et al.\ \cite{Ref029}. In what follows,  we chiefly consider such case;
on selecting $\omega^\alpha_k$ as (\ref{eq08}) and imposing $0<\alpha<1$, (\ref{eq07}) simply  reduces to %turns into %becomes
\begin{equation}\label{eq09}
   {^C_0}D^\alpha_tf(x,t_n)= \frac{1}{\tau^\alpha}\sum_{k=0}^{n}\omega^\alpha_kf(x,t_{n-k})
   -\frac{1}{\tau^\alpha}\sum_{k=0}^{n}\omega^\alpha_kf(x,0)+\mathscr{R}_\tau,
\end{equation}
with the truncated error $\mathscr{R}_\tau$ satisfying $\mathscr{R}_\tau=\mathscr{O}(\tau)$.
\begin{lemma}\label{le1}
The coefficients $\omega^\alpha_k$ defined in (\ref{eq08}) fulfill
\begin{itemize}
   \item[(a)] $\omega^\alpha_0=1, \quad \omega^\alpha_k< 0$, \ \ $\forall k\geq 1$,
   \item[(b)] $\sum_{k=0}^{\infty}\omega^\alpha_k=0, \quad \sum_{k=0}^{n-1}\omega^\alpha_k>0$.
\end{itemize}
\end{lemma}
\begin{proof}
See references \cite{Ref036,Ref032} for details.
%The detailed proof see  \cite{Ref036,Ref032}.
\end{proof}

\subsection{A fully discrete exponential B-spline based scheme}
Define $V_{M+3}=\textrm{span}\{B_{-1}(x),B_0(x),\ldots,B_{M}(x),B_{M+1}(x)\}$ over the interval $[a,b]$ referred to as a $(M+3)$-dimensional
exponential spline space. Then, an approximate solution to Eqs. (\ref{eq01})-(\ref{eq03}) is sought on $V_{M+3}$ in the form
\begin{equation}\label{eq10}
    u_N(x,t)=\sum_{j=-1}^{M+1}\alpha_j(t)B_j(x),
\end{equation}
with the unknown weights $\{\alpha_j(t)\}_{j=-1}^{M+1}$ yet to be determined by some certain restrictions.
Discretizing Eq. (\ref{eq01}) by using (\ref{eq09}) in time, we have
\begin{align*}
   u(x,t_n)-\tau^\alpha\kappa\frac{\partial^2 u(x,t_n)}{\partial x^2}=
   -\sum_{k=1}^{n-1}\omega^\alpha_ku(x,t_{n-k})+\sum_{k=0}^{n-1}\omega^\alpha_ku(x,0)+\tau^\alpha f(x,t_n)+\tau^\alpha\mathscr{R}_\tau.
\end{align*}
Let $\alpha^n_j=\alpha_j(t_n)$.
On replacing $u(x,t)$ by $u_N(x,t)$ and imposing the following collocation and boundary conditions
\begin{align*}
   &u_N(x_j,t_n)-\tau^\alpha\kappa\frac{\partial^2 u_N(x_j,t_n)}{\partial x^2} =
   -\sum_{k=1}^{n-1}\omega^\alpha_ku_N(x_j,t_{n-k})+\sum_{k=0}^{n-1}\omega^\alpha_ku_N(x_j,0)+\tau^\alpha f(x_j,t_n), \\
   &u_N(x_0,t_n)=g_1(t_n), \quad u_N(x_M,t_n)=g_2(t_n),
\end{align*}
at each nodal point $x_j$, $j=0,1,\ldots,M$, we obtain
\begin{align}\label{eq16}
   A\alpha^n_{j-1}+A'\alpha^n_j+A\alpha^n_{j+1}
   =-\sum_{k=1}^{n-1}\omega^\alpha_kP^{n-k}_j+\sum_{k=0}^{n-1}\omega^\alpha_kP_j^0+R^n_j,
\end{align}
and the boundary sets
\begin{align}
&\frac{s-ph}{2(phc-s)}\alpha^n_{-1}+\alpha^n_0+\frac{s-ph}{2(phc-s)}\alpha^n_{1}=g_1^n,\label{eq14}\\
&\frac{s-ph}{2(phc-s)}\alpha^n_{M-1}+\alpha^n_M+\frac{s-ph}{2(phc-s)}\alpha^n_{M+1}=g_2^n, \label{eq15}
\end{align}
owing to (\ref{eq10}) and (\ref{eq11})-(\ref{eq13}), with
\begin{align*}
 &A=-\tau^\alpha\kappa p^2s+\omega^\alpha_0(s-ph),\quad A'=2\tau^\alpha\kappa p^2s+2\omega^\alpha_0(phc-s),\\
 &P^{m}_j=(s-ph)\alpha^m_{j-1}+2(phc-s)\alpha^m_j+(s-ph)\alpha^m_{j+1},\quad R^n_j=2\tau^\alpha(phc-s)f_j^n.
\end{align*}
where $m=0,1,\ldots,n-1$. As a result, using Eqs. (\ref{eq14})-(\ref{eq15}) to remove the unknown variables $\alpha^n_{-1}$,
$\alpha^n_{M+1}$ in Eq.(\ref{eq16}) when $j=0$, $M$, the above system admits a linear system of algebraic equations of size $(M+1)\times(M+1)$, as below
\begin{equation}\label{eqap}
   \textbf{A}\boldsymbol{\alpha}^n=-\sum_{k=1}^{n-1}\omega^\alpha_k\textbf{B}\boldsymbol{\alpha}^{n-k}
    +\sum_{k=0}^{n-1}\omega^\alpha_k\textbf{B}\boldsymbol{\alpha}^0+\textbf{F}^n,
\end{equation}
where
\begin{align*}
     \textbf{A}=\left(
    \begin{array}{ccccc}
      2\tau^\alpha\kappa p^3hs(c-1) & 0 &   &   &   \\
      A & A' & A &  &  \\
        & \cdots & \ldots & \ldots &  \\
        &  & \ldots & \ldots & \ldots \\
        &  & A & A' & A \\
        &   &   & 0 & 2\tau^\alpha\kappa p^3hs(c-1) \\
    \end{array}
  \right),
\end{align*}
\begin{align*}
     \textbf{B}=\left(
    \begin{array}{ccccc}
      0 & 0 &   &   &   \\
      s-ph & 2(phc-s) & s-ph &  &  \\
        & \cdots & \ldots & \ldots &  \\
        &  & \ldots & \ldots & \ldots \\
        &  & s-ph & 2(phc-s) & s-ph \\
        &   &   & 0 & 0 \\
    \end{array}
  \right),
\end{align*}
\begin{align*}
\boldsymbol{\alpha}^{m}=\left(
\begin{array}{c}
  \alpha^{m}_0 \\
  \alpha^{m}_1 \\
  \vdots \\
  \alpha^{m}_{M-1} \\
  \alpha^{m}_M
\end{array}
\right), \ \
\textbf{F}^{n}=(phc-s)\left(
\begin{array}{c}
  2\tau^\alpha(s-ph)f^n_0+d^n_0 \\
  2\tau^\alpha f^n_1 \\
  \vdots \\
  2\tau^\alpha f^n_{M-1} \\
  2\tau^\alpha(s-ph)f^n_M+d^n_M
\end{array}
\right),
\end{align*}
in which, $m=0,1,\ldots,n$, and $d^n_0$, $d^n_M$ are as follows
\begin{align*}
&d^n_0=-2(s-ph)\sum_{k=0}^{n-1}\omega^\alpha_kg_1^{n-k}+2(s-ph)\sum_{k=0}^{n-1}\omega^\alpha_k\varphi_0+2\tau^\alpha\kappa p^2sg^n_1, \\
&d^n_M=-2(s-ph)\sum_{k=0}^{n-1}\omega^\alpha_kg_2^{n-k}+2(s-ph)\sum_{k=0}^{n-1}\omega^\alpha_k\varphi_M+2\tau^\alpha\kappa p^2sg^n_2.
\end{align*}

The weights $\boldsymbol{\alpha}^n$ depends on $\boldsymbol{\alpha}^{n-k}$, $k=0,1,\ldots,n$, at its previous time levels and is found via
a recursive style; once $\boldsymbol{\alpha}^n$ is obtained, $\alpha^n_{-1}$, $\alpha^n_{M+1}$  are obvious due to Eqs. (\ref{eq14})-(\ref{eq15}).
On the other side, $\textbf{A}$ is a $(M+1)\times(M+1)$ tri-diagonal matrix, therefore the system can be performed by the well-known Thomas algorithm,
which simply needs the arithmetic operation cost $\mathscr{O}(M+1)$.

\section{Initial state}\label{s4}
In order to start Eq. (\ref{eqap}), an appropriate initial vector $\boldsymbol{\alpha}^0$ to the system is required.
To this end, we employ the initial conditions
\begin{equation*}
 u_N(x_j,0)=\varphi(x_j), \quad j=0,1,\cdots,M,
\end{equation*}
together with the collocation constraints %conditions  %extra additional
\begin{equation*}
    u'_N(x_0,0)=\varphi'(x_0),\quad \ u'_N(x_M,0)=\varphi'(x_M),
\end{equation*}
got via Eq. (\ref{eq02}) explicitly to  determine a unique initial vector $\boldsymbol{\alpha}^0$ by
\begin{equation}\label{eq17}
    \textbf{K}\boldsymbol{\alpha}^0=\textbf{U},
\end{equation}
with the notations
\begin{align*}
     \textbf{K}=\left(
    \begin{array}{ccccc}
      phc-s & s-ph &   &   &   \\
      s-ph & 2(phc-s) & s-ph &  &  \\
        & \cdots & \ldots & \ldots &  \\
        &  & \ldots & \ldots & \ldots \\
        &  & s-ph & 2(phc-s) & s-ph \\
        &   &   & s-ph & phc-s \\
    \end{array}
  \right),
\end{align*}
\begin{align*}
\boldsymbol{\alpha}^{0}=\left(
\begin{array}{c}
  \alpha^{0}_0 \\
  \alpha^{0}_1 \\
  \vdots \\
  \alpha^{0}_{M-1} \\
  \alpha^{0}_M
\end{array}
\right), \ \quad
\textbf{U}=(phc-s)\left(
\begin{array}{c}
  \varphi_0-\frac{(s-ph)\varphi'(x_0)}{p(1-c)} \\
  2\varphi_1 \\
  \vdots \\
  2\varphi_{M-1} \\
  \varphi_M+\frac{(s-ph)\varphi'(x_M)}{p(1-c)}
\end{array}
\right).
\end{align*}

In the same fashion, $\textbf{K}$ is a $(M+1)\times(M+1)$ tri-diagonal matrix, so the solution
of Eq. (\ref{eq17}) can also be computed by Thomas algorithm.

\section{Stability and solvability}\label{s5}
In this section, our objective is to prove that Eqs. (\ref{eqap})-(\ref{eq17}) are uniquely solvable and unconditionally stable.
If $\tilde{\alpha}^n_j$, $n\geq1$, is a perturbed solution of Eq. (\ref{eq16}), we shall study how the perturbation
$\rho^n_j=\alpha^n_j-\tilde{\alpha}^n_j$, which solves the homogeneous counterpart of the equation by
\begin{equation}\label{eq18}
   A\rho^n_{j-1}+A'\rho^n_j+A\rho^n_{j+1}
   =-\sum_{k=1}^{n-1}\omega^\alpha_kZ^{n-k}_j+\sum_{k=0}^{n-1}\omega^\alpha_kZ_j^0,
\end{equation}
evolves over time, where $Z^{0}_j$, $Z^{n-k}_j$ are the quantities like $P^{0}_j$, $P^{n-k}_j$ with regard to the perturbation.
Since the classic von Neumann method does not work for Eq. (\ref{eq18}), a fractional procedure is employed to analyze its stability.
This extension was recently laid down in \cite{Ref031} applied to discuss a non-uniform implicit difference scheme
for fractional diffusion equations.

\begin{lemma}
The system (\ref{eqap})-(\ref{eq17}) are uniquely solvable since their coefficient matrices $\textbf{A}$, $\textbf{K}$ are
strictly diagonally dominant.
\end{lemma}
\begin{proof}
In virtue of $A$, $A'$, one gets
\begin{align*}
\big|A'\big|-2\big|A\big|&=2\big|\tau^\alpha\kappa p^2s+\omega^\alpha_0(phc-s)\big|-2\big|-\tau^\alpha\kappa p^2s+\omega^\alpha_0(s-ph)\big|\\
 &\geq 2\omega^\alpha_0(phc-s)-2\omega^\alpha_0(s-ph),\\
 &=2\omega^\alpha_0\big((phc-s)-(s-ph)\big).
\end{align*}
Then, the lemma is ascribed to $s-ph<phc-s$. Using the following Taylor's expansions
\begin{align*}
   &s-ph=\frac{(ph)^3}{3!}+\frac{(ph)^5}{5!}+\cdots+\frac{(ph)^{2k+1}}{(2k+1)!}+\cdots \\
   &phc-ph=\frac{(ph)^3}{2!}+\frac{(ph)^5}{4!}+\cdots+\frac{(ph)^{2k+1}}{(2k)!}+\cdots
\end{align*}
results in %leads to
\begin{align*}
   (phc-ph)-2(s-ph)&=(ph)^3\Bigg(\frac{1}{2!}-\frac{2}{3!}\Bigg)+(ph)^5\Bigg(\frac{1}{4!}-\frac{2}{5!}\Bigg)\\
      &\quad+\cdots+(ph)^{2k+1}\Bigg(\frac{1}{(2k)!}-\frac{2}{(2k+1)!}\Bigg)+\cdots
\end{align*}
Due to $(2k)!\times2<(2k)!\times(2k+1)$, $k\geq1$, there exist
\begin{equation*}
  (phc-s)-(s-ph)=(phc-ph)-2(s-ph)>0,
\end{equation*}
and $\big|A'\big|-2\big|A\big|>0$, which implies $\textbf{A}$ is strictly diagonally dominant, so is $\textbf{K}$.
Hence, Eqs. (\ref{eqap})-(\ref{eq17}) are uniquely solvable. The proof is completed.
\end{proof}

The stable analysis is proceeded as following.
\begin{theorem}
The system (\ref{eqap})-(\ref{eq17}) are unconditionally stable.
\end{theorem}
\begin{proof}
As the usual way, we investigate a single generic mode $\rho^k_j=\zeta^k_\upsilon\exp(\textrm{i}\upsilon jh)$,
with $\textrm{i}=\sqrt{-1}$ and the wave number $\upsilon$.
Inserting it into Eq. (\ref{eq18}) yields
\begin{align*}
   2A\zeta^n_\upsilon\cos(\upsilon h)+A'\zeta^n_\upsilon
    =-\sum_{k=1}^{n-1}\omega^\alpha_kS^{n-k}_\upsilon
    + \sum_{k=0}^{n-1}\omega^\alpha_kS^{0}_\upsilon,
\end{align*}
where
\begin{align*}
&S^{0}_\upsilon=2(s-ph)\cos(\upsilon h)\zeta^{0}_\upsilon +2(phc-s)\zeta^{0}_\upsilon,\\
&S^{n-k}_\upsilon=2(s-ph)\cos(\upsilon h)\zeta^{n-k}_\upsilon+2(phc-s)\zeta^{n-k}_\upsilon,
\end{align*}
%$S^{0}_\upsilon=2(s-ph)\cos(\upsilon h)\zeta^{0}_\upsilon +2(phc-s)\zeta^{0}_\upsilon$,
%$S^{n-k}_\upsilon=2(s-ph)\cos(\upsilon h)\zeta^{n-k}_\upsilon+2(phc-s)\zeta^{n-k}_\upsilon$,
by the aid of Euler's formula $\exp(\pm\textrm{i}\upsilon h)=\cos(\upsilon h)\pm\textrm{i}\sin(\upsilon h)$.
Noticing that
\begin{equation*}
2A\cos(\upsilon h)+A'=2\tau^\alpha\kappa p^2s(1-\cos(\upsilon h))
    +2\omega^\alpha_0(s-ph)\cos(\upsilon h)+2\omega^\alpha_0(phc-s),
\end{equation*}
and the inequalities
\begin{equation*}
    s-ph>0, \quad phc-s>0, \quad s-ph<phc-s,
\end{equation*}
we obtain
\begin{align}\label{eq19}
 \zeta^n_\upsilon=-\sum_{k=1}^{n-1}\omega^\alpha_kG\zeta^{n-k}_\upsilon
  +\sum_{k=0}^{n-1}\omega^\alpha_kG\zeta^{0}_\upsilon,
\end{align}
with a fixed quantity
\begin{equation*}
G=\frac{\omega^\alpha_0(s-ph)\cos(\upsilon h)+\omega^\alpha_0(phc-s)}{\tau^\alpha\kappa p^2s(1-\cos(\upsilon h))
    +\omega^\alpha_0(s-ph)\cos(\upsilon h)+\omega^\alpha_0(phc-s)},
\end{equation*}
not more than $1$. \!To show $|\zeta^n_\upsilon|\leq|\zeta^0_\upsilon|$, we use mathematical induction. As $n=1$, by Eq. (\ref{eq19}), we trivially have
$|\zeta^1_\upsilon|\leq|\zeta^{0}_\upsilon|$, since $\omega^\alpha_0G\leq1$. Assuming that
%\begin{equation}
%    |\zeta^k_\upsilon|\leq \max\big\{|\zeta^0_\upsilon|,|\zeta^1_\upsilon|,\ldots,|\zeta^{k-1}_\upsilon|\big\}, \quad k=1,2,\ldots,n-1,
%\end{equation}
\begin{equation}\label{eq20}
    |\zeta^m_\upsilon|\leq |\zeta^0_\upsilon|, \quad m=1,2,\ldots,n-1,
\end{equation}
it follows from Lemma \ref{le1} that
\begin{align*}
 |\zeta^n_\upsilon|&\leq\Bigg|-\sum_{k=1}^{n-1}\omega^\alpha_kG\zeta^{n-k}_\upsilon
  +\sum_{k=0}^{n-1}\omega^\alpha_kG\zeta^{0}_\upsilon\Bigg| \\
  &\leq\Bigg(1-\sum_{k=0}^{n-1}\omega^\alpha_k+\sum_{k=0}^{n-1}\omega^\alpha_k\Bigg)
   G\max\limits_{0\leq m\leq n-1}|\zeta^m_\upsilon| \\
  &=G\max\limits_{0\leq m\leq n-1}|\zeta^m_\upsilon|,
\end{align*}
which implies $|\zeta^n_\upsilon|\leq |\zeta^0_\upsilon|$ for $G<1$ and the assumption (\ref{eq20}).
Hence, we realize that the perturbation remains  bounded
by its initial perturbation unconditionally at any time level. This proves what is required.
\end{proof}

\section{Numerical experiments}\label{s6}
In this part, the proposed exponential B-spline collocation method is tested on a couple of numerical examples, which suffice to
gauge its accuracy and realistic performance. The computed errors are measured by $L^2$- and $L^\infty$-norms, i.e.,
\begin{align*}
&||u(x,t)-u_N(x,t)||_{L^2}=\sqrt{h\sum^{M-1}_{j=1}\Big|u(x_j,t)-u_N(x_j,t)\Big|^2},\\
&||u(x,t)-u_N(x,t)||_{L^\infty}=\max\limits_{1\leq j\leq M-1}\Big|u(x_j,t)-u_N(x_j,t)\Big|,
\end{align*}
and for every concrete problem, the tension parameter $p$ is optimally selected.
The resulting algebraic equations are handled by Thomas algorithm and the numerical results may be compared
with the other existent methods. \\

\noindent
\textbf{Example 6.1.}
Let $a=0$, $b=1$, $T=1$, and the initial boundary conditions
$\varphi(x)=0$, $g_1(t)=0$, $g_2(t)=0$.
The right side is given as
\begin{equation*}
f(x,t)=\frac{\Gamma(1+\alpha)}{\Gamma(\mu+1-\alpha)}t^{\mu-\alpha}x^3(1-x)-6\kappa t^\mu x(1-2x),
\end{equation*}
to enforce the exact solution $u(x,t)=t^\mu x^3(1-x)$. Taking $\kappa=1$, $\mu=2+\alpha$, $p=1.18$, the algorithm is run on
the meshes using collocation numbers $M=128$, $N=3200$, and $M=256$, $N=6400$,
with various fractional differentiation $\alpha$. Table \ref{tab1} reports the absolute errors at several nodal points when $t=T$.
It is obvious that the method is considerably robust and accurate. \\

\begin{table*}%[!htb]
\centering
\caption{The absolute errors at some nodal points with $p=1.18$ and various $\alpha$ for Example 6.1} \label{tab1}
\begin{tabular}{lllllll}
\hline
\multicolumn{1}{l}{\multirow{2}{0.6cm}{x}}
&\multicolumn{3}{l}{$M=128$, $N=3200$} &\multicolumn{3}{l}{$M=256$, $N=6400$} \\
\cline{2-7}& $\alpha=0.3$  &$\alpha=0.6$ & $\alpha=0.9$  &$\alpha=0.3$  &$\alpha=0.6$ & $\alpha=0.9$ \\
\hline  0.1     &4.8077e-6  &4.9205e-6  &  5.1822e-6  &1.2909e-6  &   1.3984e-6  &1.6069e-6 \\
        0.2     &8.8365e-6  &9.0944e-6  &  9.6736e-6  &2.3330e-6  &   2.5585e-6  &2.9931e-6   \\
        0.3     &1.1667e-5  &1.2103e-5  &  1.3046e-5  &3.0516e-6  &   3.3966e-6  &4.0561e-6   \\
        0.4     &1.3323e-5  &1.3957e-5  &  1.5284e-5  &3.5061e-6  &   3.9726e-6  &4.8565e-6   \\
        0.5     &1.3817e-5  &1.4637e-5  &  1.6304e-5  &3.6600e-6  &   4.2262e-6  &5.2909e-6   \\
        0.6     &1.3252e-5  &1.4187e-5  &  1.6052e-5  &3.5199e-6  &   4.1375e-6  &5.2916e-6   \\
        0.7     &1.1547e-5  &1.2493e-5  &  1.4349e-5  &3.0707e-6  &   3.6688e-6  &4.7805e-6   \\
        0.8     &8.7289e-6  &9.5223e-6  &  1.1061e-5  &2.3479e-6  &   2.8378e-6  &3.7446e-6   \\
        0.9     &4.8491e-6  &5.3116e-6  &  6.2017e-6  &1.3059e-6  &   1.5864e-6  &2.1042e-6  \\
\hline
\end{tabular}
\end{table*}

\noindent
\textbf{Example 6.2.}
Recalling the Mittag-Leffler function
\begin{equation*}
    E_\alpha(z)=\sum_{k=0}^{\infty}\frac{z^k}{\Gamma(\alpha k+1)}, \quad 0<\alpha<1,
\end{equation*}
endowed with  ${^C_0}D^\alpha_tE_\alpha(-\lambda t^\alpha)=-\lambda E_\alpha(-\lambda t^\alpha)$ \cite{Ref069},
we consider Eqs. (\ref{eq01})-(\ref{eq03}) on domain (0,1) with
\begin{equation*}
  u(x,0)=\sin(\pi x/2), \quad g_1(t)=0, \quad g_2(t)=E_\alpha(- t^\alpha),
\end{equation*}
and homogeneous force term. It is easy to verify that its exact solution takes the form
$u(x,t)=E_\alpha(- t^\alpha)\sin(\pi x/2)$, when $\kappa=4/\pi^2$.
On collocating the domains by setting $M=50$, $N=2500$, and $M=100$, $N=10000$,
the numerical results corresponding to $p=1.52$ at $t=1$ are tabulated in Table \ref{tab2}, where we observe that
the proposed method is quite stable and accurate. \\

\begin{table*}%[!htb]
\centering
\caption{The absolute errors at some nodal points with $p=1.52$ and various $\alpha$ for Example 6.2} \label{tab2}
\begin{tabular}{lllllll}
\hline
\multicolumn{1}{l}{\multirow{2}{0.6cm}{x}}
&\multicolumn{3}{l}{$M=50$, $N=2500$} &\multicolumn{3}{l}{$M=100$, $N=10000$} \\
\cline{2-7}& $\alpha=0.3$  &$\alpha=0.6$ & $\alpha=0.9$  &$\alpha=0.3$  &$\alpha=0.6$ & $\alpha=0.9$ \\
\hline  0.1     &2.6511e-6  &1.7151e-6  &1.3626e-7  &6.6926e-07  &  4.3003e-7  &3.3909e-8\\
        0.2     &5.1402e-6  &3.3299e-6  &2.5439e-7  &1.2977e-06  &  8.3493e-7  &6.3294e-8 \\
        0.3     &7.3057e-6  &4.7433e-6  &3.3856e-7  &1.8445e-06  &  1.1893e-6  &8.4212e-8 \\
        0.4     &8.9870e-6  &5.8526e-6  &3.7746e-7  &2.2693e-06  &  1.4674e-6  &9.3849e-8 \\
        0.5     &1.0024e-5  &6.5525e-6  &3.6624e-7  &2.5317e-06  &  1.6429e-6  &9.0999e-8 \\
        0.6     &1.0259e-5  &6.7351e-6  &3.0810e-7  &2.5915e-06  &  1.6886e-6  &7.6468e-8 \\
        0.7     &9.5339e-6  &6.2885e-6  &2.1542e-7  &2.4088e-06  &  1.5766e-6  &5.3358e-8 \\
        0.8     &7.6895e-6  &5.0973e-6  &1.1038e-7  &1.9433e-06  &  1.2779e-6  &2.7210e-8 \\
        0.9     &4.5661e-6  &3.0421e-6  &2.4885e-8  &1.1543e-06  &  7.6266e-7  &6.0045e-9 \\
\hline
\end{tabular}
\end{table*}

\noindent
\textbf{Example 6.3.}
In this test, we consider a special case of $\alpha=0.5$. Let $a=0$, $b=1$, $\kappa=1$, $\varphi(x)=\cos(6\pi x)$,
$g_1(t)=\textrm{erfcx}(36\pi^2\sqrt{t})$, $g_2(t)=g_1(t)$, $f(x,t)=0$, and the true solution (see \cite{Ref033})
\begin{equation*}
    u(x,t)=\cos(6\pi x)\textrm{erfcx}(36\pi^2\sqrt{t}),
\end{equation*}
where $\textrm{erfcx}(\cdot)$ is the \emph{scaled complementary error function}, given by
\begin{equation*}
    \textrm{erfcx}(z)=\frac{2}{\sqrt{\pi}}\exp(z^2)\int_{z}^\infty \exp{(-\eta^2)}d\eta.
\end{equation*}
%The computation is first run with fixed tension parameter $p=0.01$.
%Fig. \ref{fig1} depicts the numerical solutions at different $t$ contrasted to the exact solutions,
%which shows that the numerical solutions are in good agreement with the exact ones.
%Table \ref{tab3} lists the $L^2$- and $L^\infty$- global errors at $t=1$, $t=2$, and $t=3$ with various
%$M$, $N$, while Table \ref{tab4} reports the global errors for various $p$ with fixed  $M=30$, $N=900$.
%%It is seen from those tables that the free parameter $p$ can affect the accuracy of our method. \\
%It is  visible that our method well approximates the problem as expected and
%$p$ is a vital quantity that adjusts its accuracy. \\
The computation is run with $p=0.01$. Fig. \ref{fig1} describes the numerical solutions at different time
compared to the exact solutions when $M=100$, $N=500$. As the graph shows, the exact and numerical solutions are in good agreement.
Table \ref{tab3} reports the global errors at $t=1$, $t=2$, and $t=3$ with various %$L^2$- and $L^\infty$-
$M$, $N$. It is visible that Eqs. (\ref{eqap})-(\ref{eq17}) well solve the test problem as expected. \\

\begin{figure}[!htb]
\centering
\includegraphics[width=4.5in]{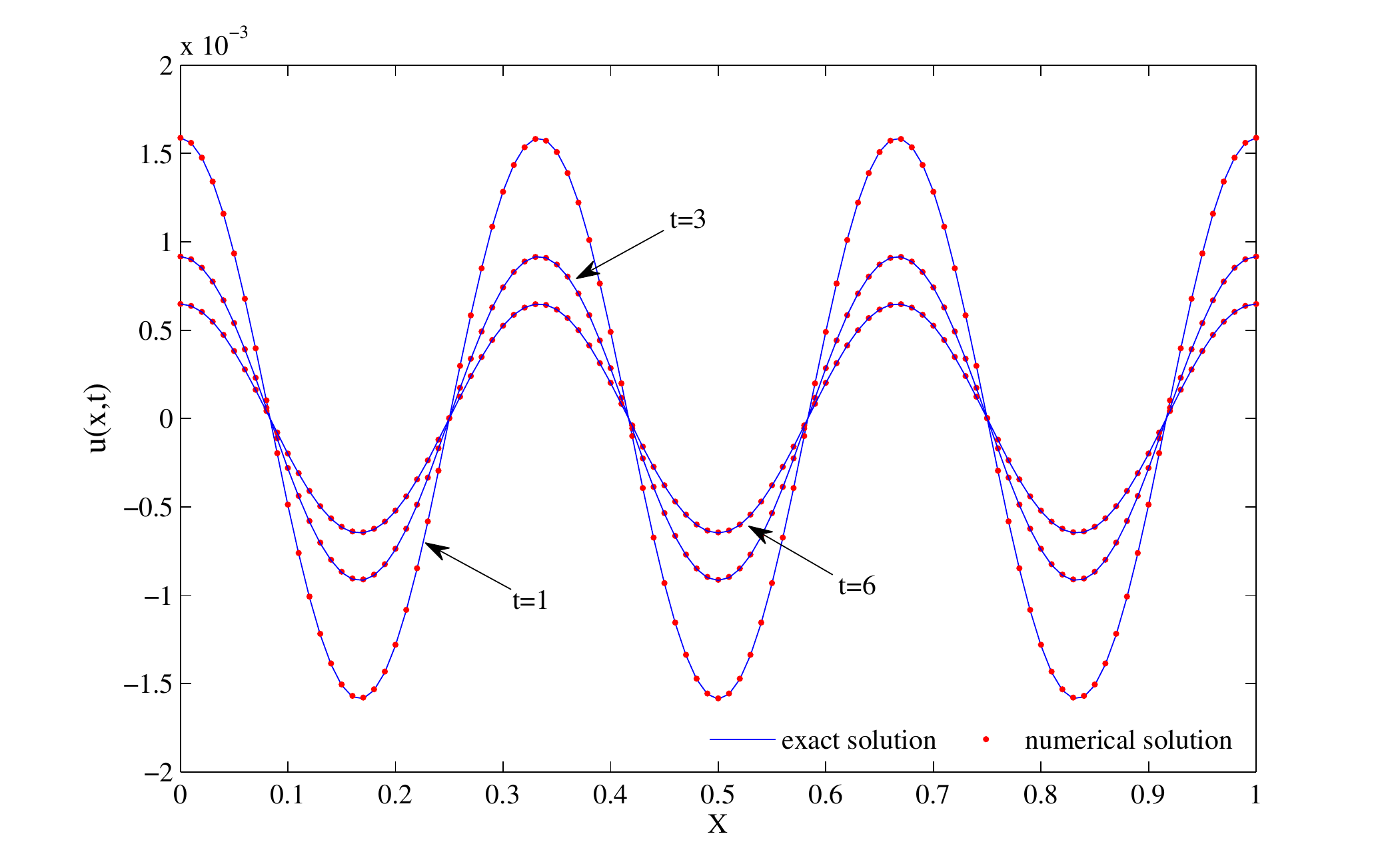}
\caption{The exact and numerical solutions at $t=1$, $3$, and $6$, when $M=100$, $N=500$.}\label{fig1}
\end{figure}

\begin{table*}[!htb]
\centering
\caption{The global errors at different time with $p=0.01$ and various $M$, $N$ for Example 6.3} \label{tab3}
\begin{tabular}{lllllll}
\hline
\multicolumn{1}{l}{\multirow{2}{1cm}{$M$, $N$}}
&\multicolumn{3}{l}{$||u-u_N||_{L^2}$} &\multicolumn{3}{l}{$||u-u_N||_{L^\infty}$}\\
\cline{2-7}& $t=1$  &$t=2$ &$t=3$  &$t=1$  &$t=2$ &$t=3$ \\
\hline  32, 4000    &5.4324e-5  &3.8735e-5  &3.1716e-5  &8.8587e-5  &  6.3211e-5  &5.1768e-5 \\
        64, 4000    &1.3203e-5  &9.6132e-6  &7.9258e-6  &2.1878e-5  &  1.5795e-5  &1.2985e-5 \\
        128, 9000   &3.0826e-6  &2.3273e-6  &1.9418e-6  &5.2449e-6  &  3.8791e-6  &3.2140e-6 \\
        256, 9000   &5.3117e-7  &4.7773e-7  &4.2641e-7  &9.5837e-7  &  8.4372e-7  &7.3492e-7 \\
        1024, 250   &5.9928e-6  &2.1116e-6  &1.1412e-6  &9.4652e-6  &  3.3298e-6  &1.7970e-6 \\
        1024, 500   &3.6171e-6  &1.2685e-6  &6.8189e-7  &5.6050e-6  &  1.9593e-6  &1.0500e-6 \\
        2048, 1000  &2.2589e-6  &7.9847e-7  &4.3255e-7  &3.4336e-6  &  1.2092e-6  &6.5273e-7 \\
        2048, 2000  &1.4133e-6  &4.9781e-7  &2.6867e-7  &2.1044e-6  &  7.3642e-7  &3.9486e-7 \\
\hline
\end{tabular}
\end{table*}

%\begin{table*}[!htb]
%\centering
%\caption{The global errors with various $p$ and $M=30$, $N=900$ for Example 6.3} \label{tab4}
%\begin{tabular}{lllllll}
%\hline
%Error Type   &$p=0.001$ &$p=0.01$  &$p=1$  &$p=3$  &$p=6$  &$p=10$ \\
%\hline  $L^2$-error      &6.0379e-5  &6.0379e-5  &6.0548e-5  &6.1900e-5  &6.6447e-5  &7.7117e-5 \\
%        $L^\infty$-error &9.9581e-5  &9.9582e-5  &9.9858e-5  &1.0206e-4  &1.0950e-4  &1.2694e-4 \\
%\hline
%\end{tabular}
%\end{table*}

\noindent
\textbf{Example 6.4.}
Let $\kappa=2$, $T=1$, $\varphi(x)=0$, $g_1(t)=0$, $g_2(t)=g_1(t)$, and the force function
\begin{align*}
f(x,t)=\frac{2t^{2-\alpha}x(1-x)\exp(x)}{\Gamma(3-\alpha)}+2t^2x(x+3)\exp(x);
\end{align*}
we consider Eqs. (\ref{eq01})-(\ref{eq03}) on domain (0,1) solved by Eqs. (\ref{eqap})-(\ref{eq17})
and the cubic B-spline collocation method (CBSCM) \cite{Ref025}. The exact solution of the model is $u(x,t)=t^2x(1-x)\exp(x)$.
In Fig. \ref{fig3}, we display their absolute error distributions at $t=T$ when $\alpha=0.6$, $M=50$, $N=2500$ by taking
$p=1.45$, $2.35$, $2.53$, and $3.35$, respectively. In line with the graphs, we then choose $p=2.53$ and show a comparison
of their absolute errors at some nodal points detailedly in Table \ref{tab5}, where the accuracy of our
method is found to be overall better than CBSCM. In Fig. \ref{fig4}, we plot the global errors versus
the variation of mesh size $1/M$ in log-log scale, with $\alpha=0.6$, $p=2.53$, and $N=11000$,
which demonstrates the convergent orders of these methods are all basically of order $2$. \\

\begin{figure}
\begin{minipage}[t]{0.49\linewidth}
\includegraphics[width=2.3in]{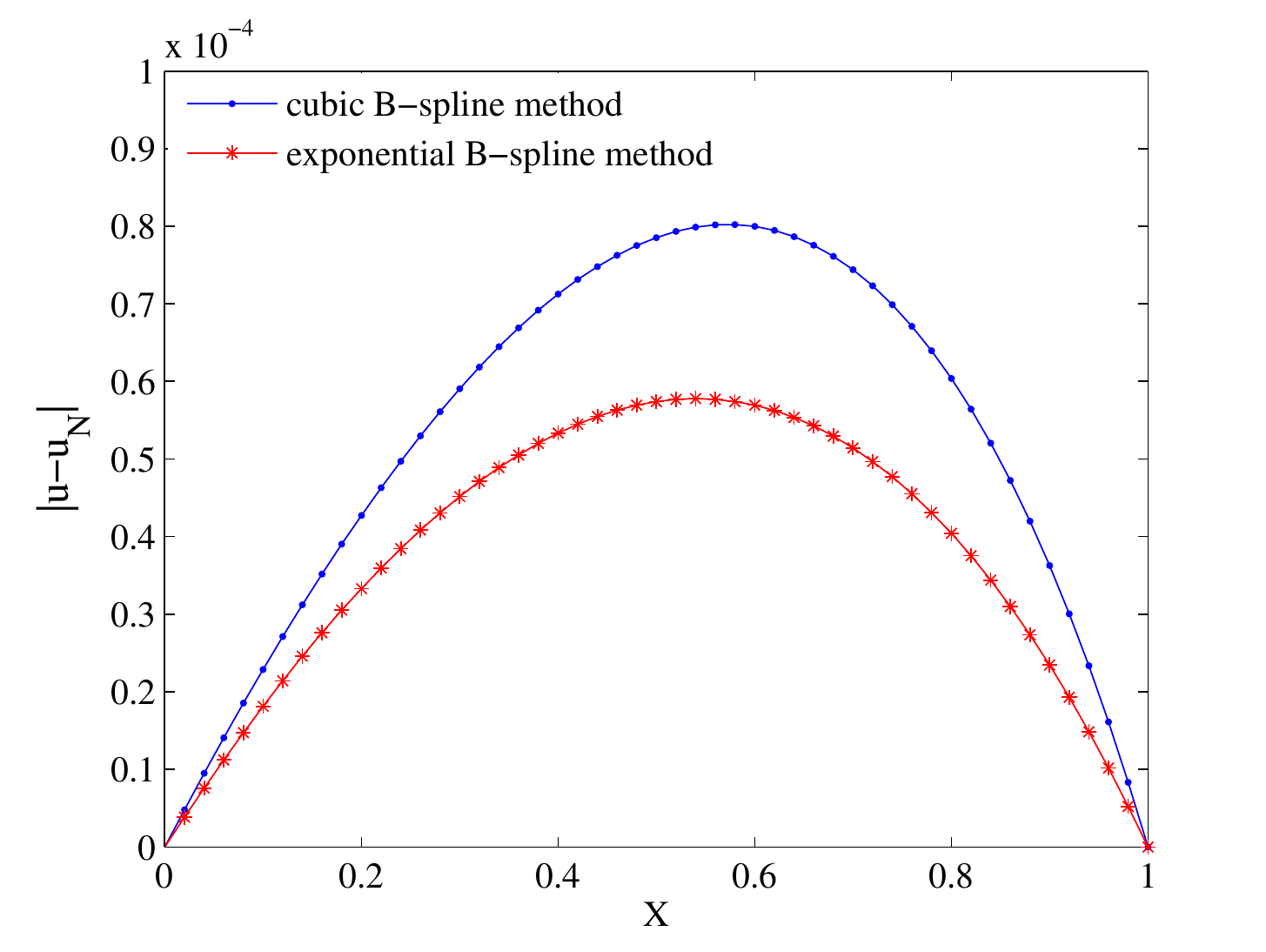}
\end{minipage}
\begin{minipage}[t]{0.5\linewidth}
\includegraphics[width=2.3in]{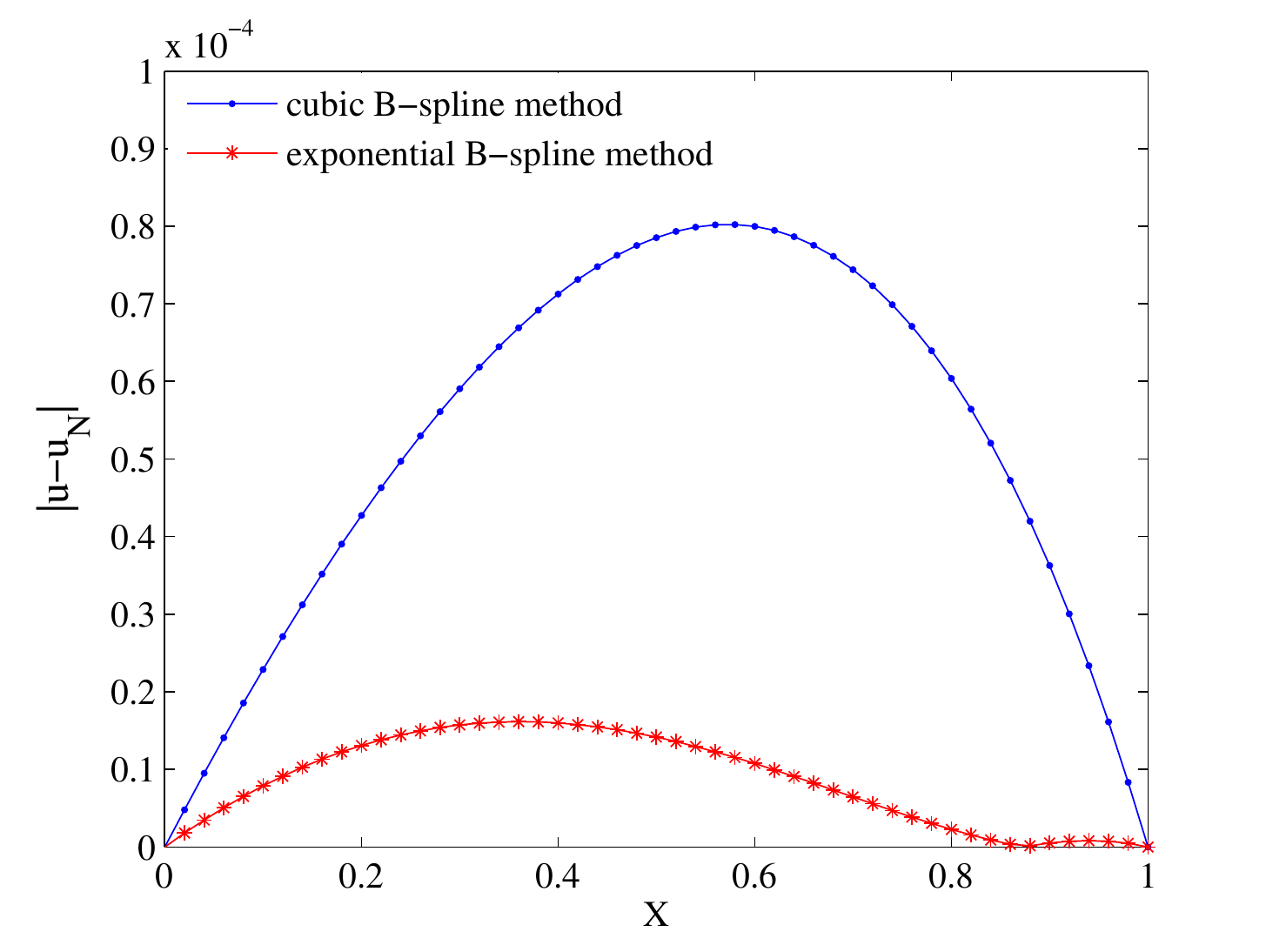}
\end{minipage}
\begin{minipage}[t]{0.49\linewidth}
\includegraphics[width=2.3in]{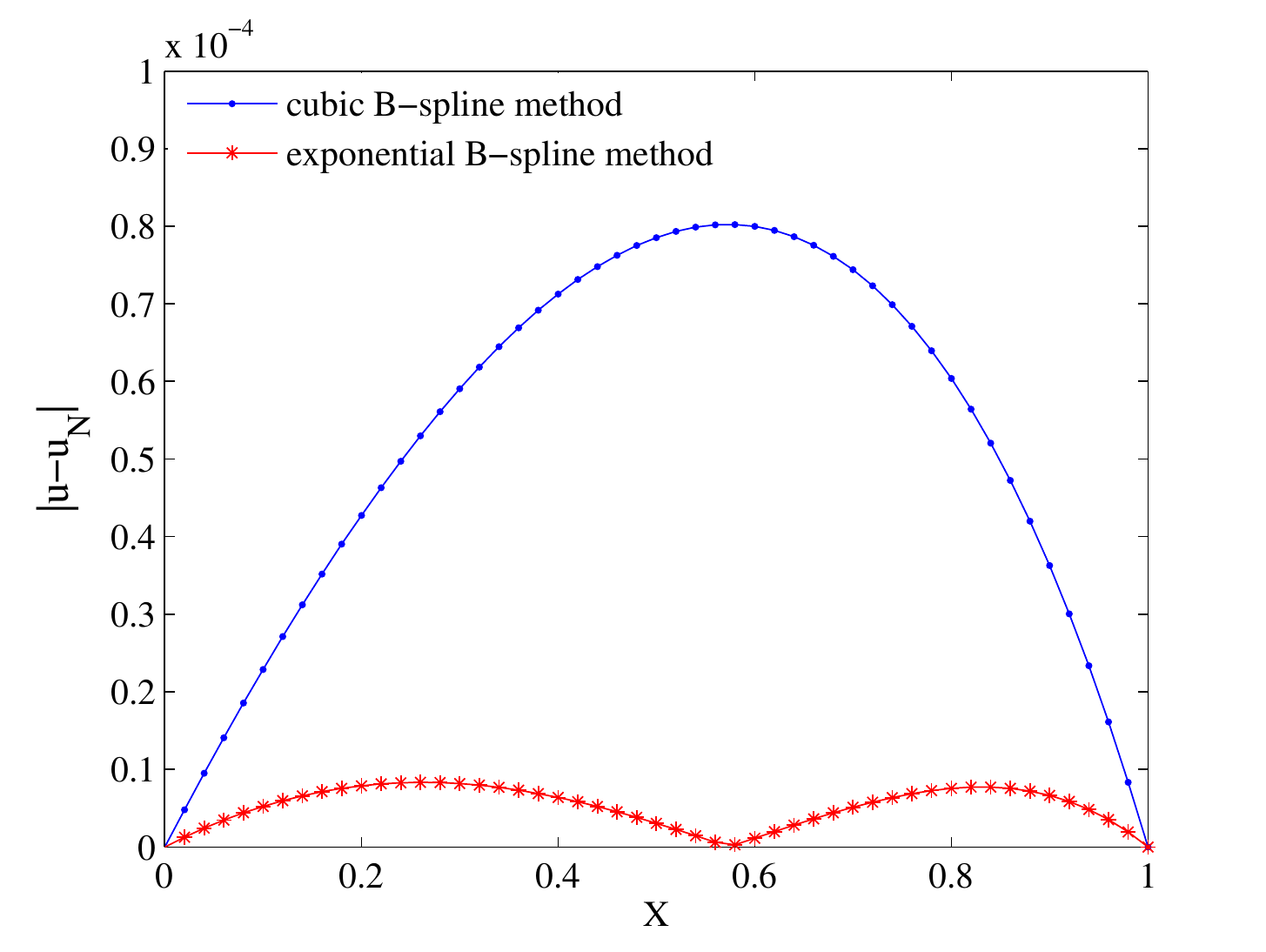}
\end{minipage}
\begin{minipage}[t]{0.49\linewidth}
\includegraphics[width=2.3in]{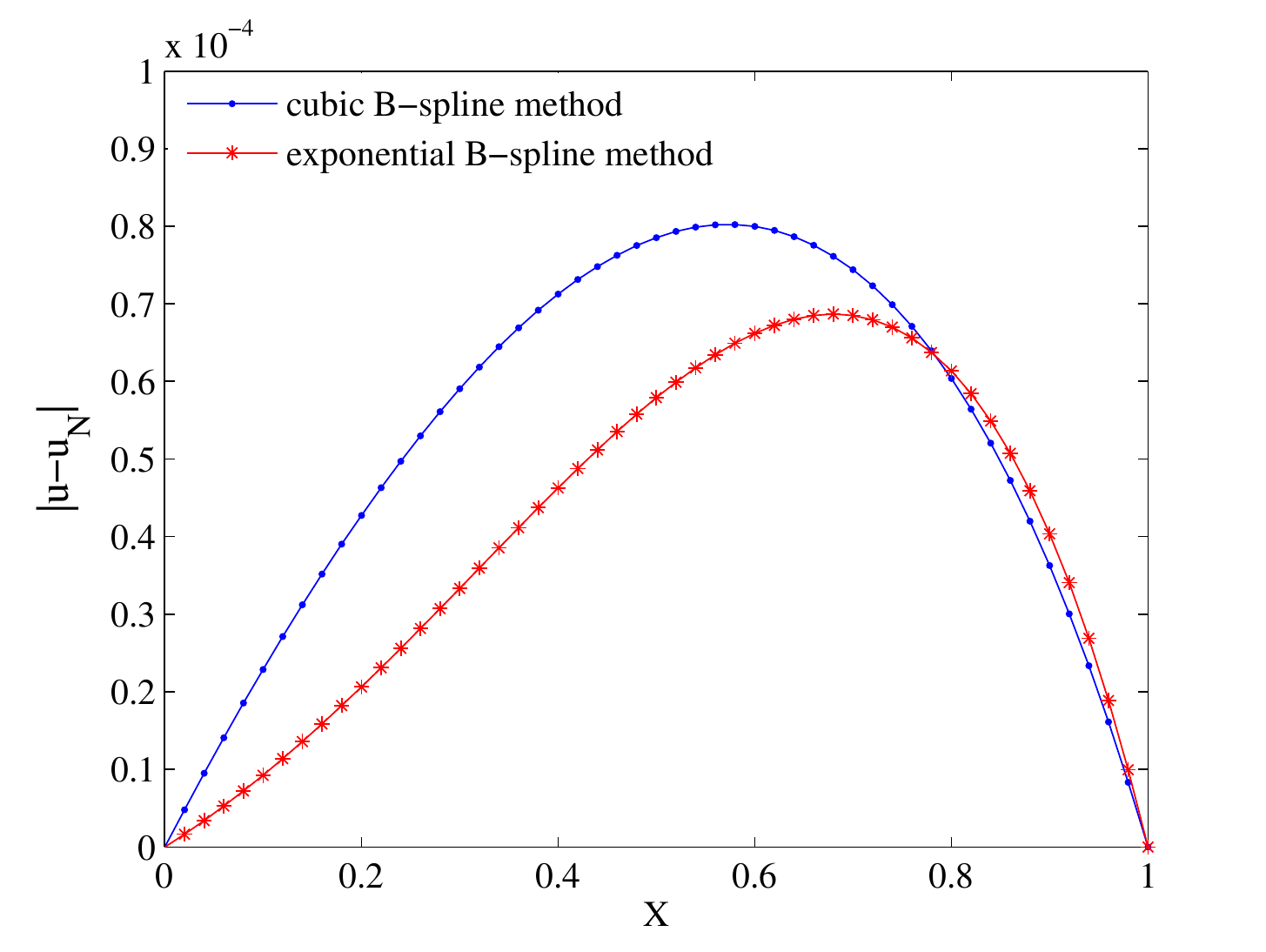}
\end{minipage}
\caption{The absolute error distributions for $p=1.45$, $2.35$, $2.53$, and $3.35$ when $M=50$, $N=2500$.}\label{fig3}
\end{figure}

\begin{table*}%[!htb]
\centering
\caption{The comparison of absolute errors between CBSCM and our method when $p=2.53$.} \label{tab5}
\begin{tabular}{lllll}
\hline
\multicolumn{1}{l}{\multirow{2}{0.6cm}{x}}
&\multicolumn{2}{l}{$M=25$, $N=625$} &\multicolumn{2}{l}{$M=50$, $N=2500$} \\
\cline{2-5}&CBSCM   &our method &CBSCM   &our method \\
\hline  0.1     &7.4297e-5  &1.7521e-5  &2.2881e-5  &5.2238e-6  \\
        0.2     &1.7128e-4  &3.1447e-5  &4.2725e-5  &7.8796e-6   \\
        0.3     &2.2488e-4  &3.3028e-5  &5.9053e-5  &8.1580e-6   \\
        0.4     &2.8563e-4  &2.5425e-5  &7.1249e-5  &6.3822e-6  \\
        0.5     &3.1076e-4  &1.5134e-5  &7.8544e-5  &3.0497e-6  \\
        0.6     &3.2060e-4  &4.5617e-6  &7.9982e-5  &1.1163e-6  \\
        0.7     &3.0518e-4  &1.7614e-5  &7.4401e-5  &5.1068e-6  \\
        0.8     &2.4201e-4  &3.0270e-5  &6.0392e-5  &7.5532e-6  \\
        0.9     &1.6825e-4  &2.8820e-5  &3.6264e-5  &6.6400e-6  \\
\hline
\end{tabular}
\end{table*}

\noindent
\textbf{Example 6.5.}
In the last test, we consider the fractional heat transfer model on $(0,1)$ with $\kappa=1$, $T=1$,
$\varphi(x)=0$, $g_1(t)=0$, and $g_2(t)=H(t-0.2)-H(t-0.6)$, where $H(\cdot)$ denotes %numerical differentiation report illustrate graph
Heaviside Step Function. As in \cite{Ref037}, the heat flux at the boundary point $x=0$ approximated by forward difference is of particular interest
and the computed results are compared with the ones obtained by implicit finite difference method in the literature.
Taking $p=1$, $M=500$, $N=125$, Fig. \ref{fig5} exhibits the heat flux at $x=0$ changing over the time for $\alpha=0.1$, $0.5$, and $0.9$.
It is clear that the results of these methods in presence are highly consistent, which reveals that our method precisely captures the heat flux. \\
%Moreover, the boundary heat flux varies continuously as a function of $\alpha$.

\section{Conclusion}
In this research, an efficient exponential B-spline based collocation method is proposed to tackle the
diffusion equation with a time-fractional derivative in Caputo sense discretized by a GMMP scheme.
It leads to a linear system of algebraic equations with tri-diagonal coefficient matrix and thereby can be solved speedily by Thomas algorithm.
The solvability is strictly evaluated and the stable analysis is proceeded by adopting a fractional von Neumann procedure.
Its codes are tested on several given models and the numerical results validate that this method is capable of dealing with these equations very well.
The comparison with the other methods manifests its practicability and advantages.
Moreover, the derived method is easy and economical to carry out, so it can be served as
an alternative choice to model the other fractional problems. \\

\begin{figure}
\begin{minipage}[t]{0.5\linewidth}
\includegraphics[width=2.4in]{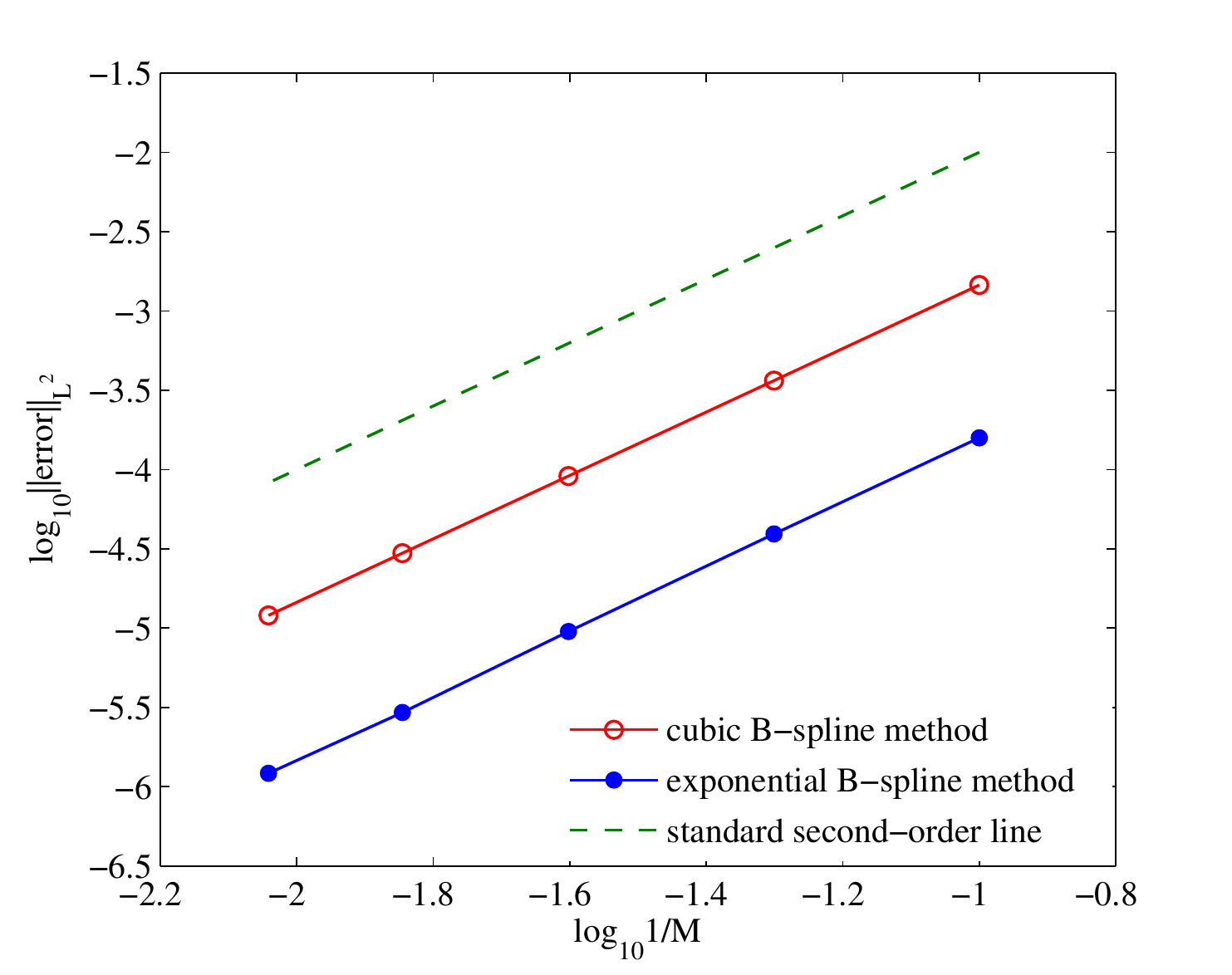}
\end{minipage}
\begin{minipage}[t]{0.5\linewidth}
\includegraphics[width=2.4in]{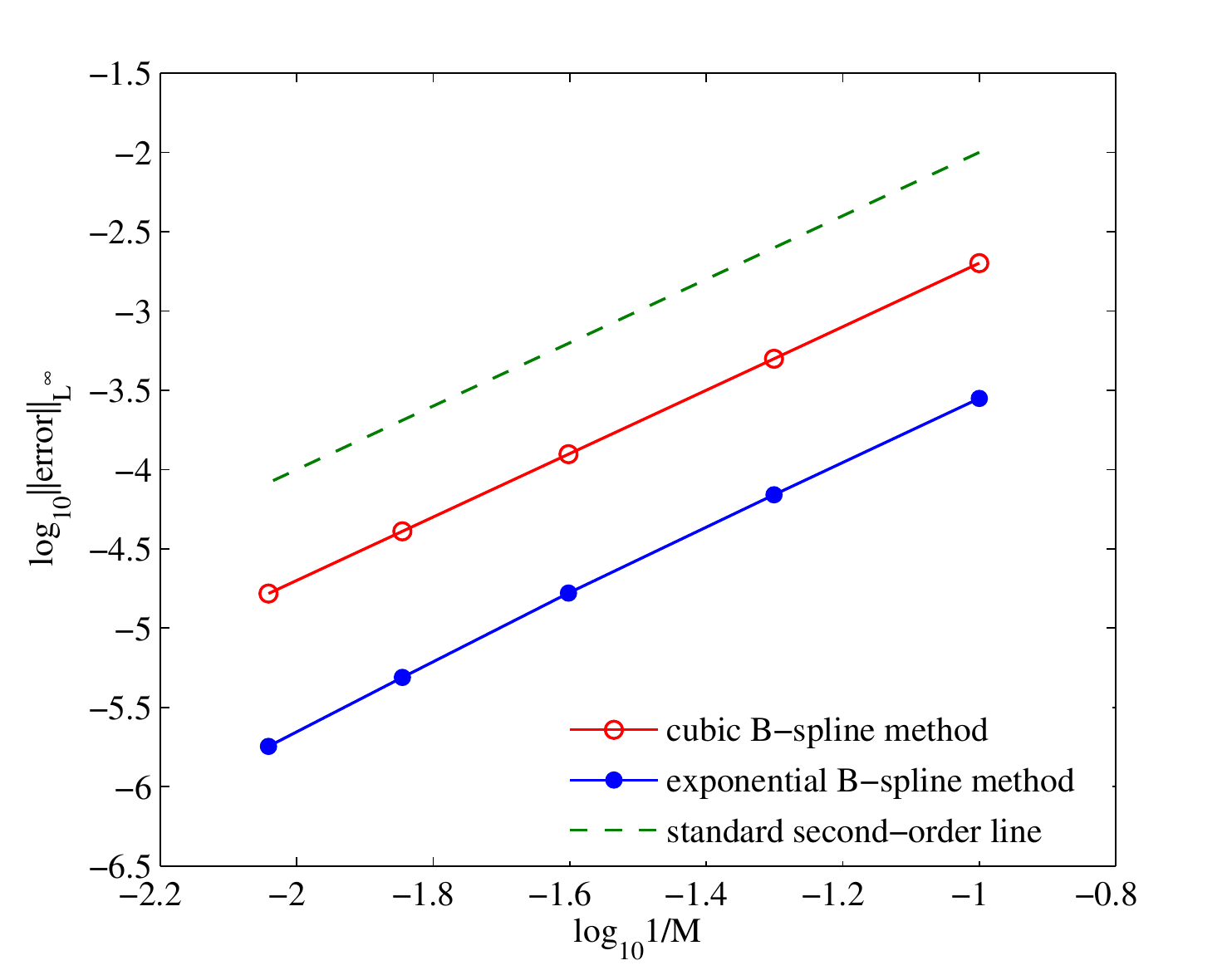}
\end{minipage}
\caption{The convergent orders of the methods with $\alpha=0.6$, $p=2.53$, and $N=11000$.}\label{fig4}
\end{figure}

\begin{figure}
\centering
\includegraphics[width=4.0in]{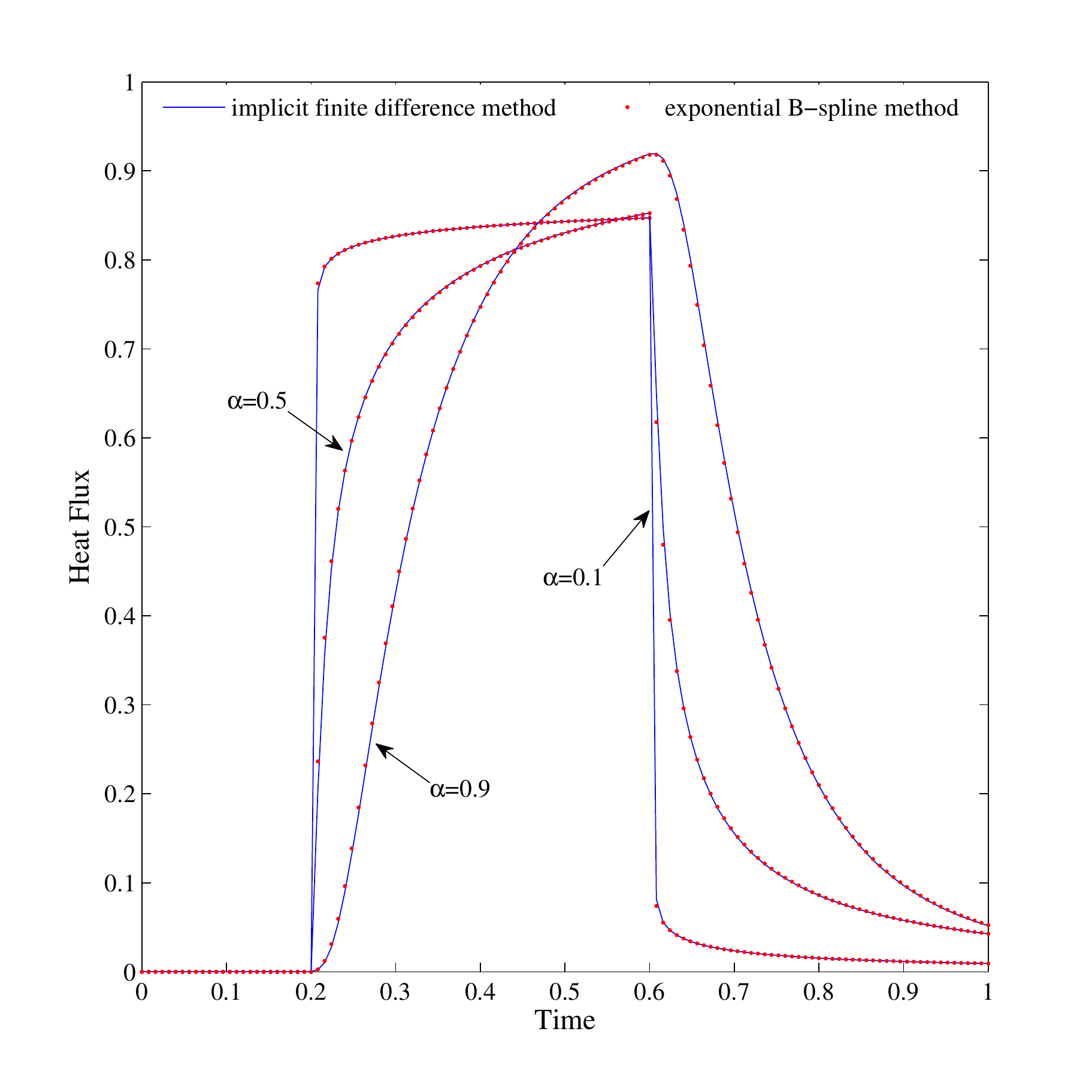}
\caption{The heat flux at $x=0$ for various $\alpha$ with $M=500$, $N=125$.}\label{fig5}
\end{figure}

\noindent
\textbf{Acknowledgement}:
This research was supported by National Natural Science Foundations of China (No.11471262 and 11501450).

\bibliographystyle{model1b-num-names}
\bibliography{mybib}

\end{document}